\documentclass[10pt]{article}
\usepackage{color}
\usepackage{amssymb}
\usepackage{amsthm,array,amssymb,amscd,amsfonts,latexsym, url}
\usepackage{amsmath}
\usepackage{verbatim}
\usepackage{tikz-cd}
\usepackage{enumitem}

\usepackage[all]{xy}
\newtheorem{theo}{Theorem}[section]
\newtheorem{thm}{Theorem}[section]
\newtheorem{prop}[theo]{Proposition}

\newtheorem{lemm}[theo]{Lemma}
\newtheorem{lemma}[theo]{Lemma}
\newtheorem{coro}[theo]{Corollary}

\newtheorem{question}[theo]{Question}

\theoremstyle{remark}
\newtheorem{rema}[theo]{Remark}
\newtheorem{rmk}[theo]{Remark}

\newcommand{\BC}{{\mathbb{C}}}

\newcommand{\BN}{{\mathbb{N}}}

\newcommand{\BQ}{{\mathbb{Q}}}

\newcommand{\BZ}{{\mathbb{Z}}}

\newcommand{\CO}{{\mathcal O}}

\newcommand{\CZ}{{\mathcal Z}}

\newcommand{\Fq}{{\mathfrak{q}}}

\newcommand{\fG}{{\mathfrak{G}}}

\newcommand{\ch}{\mathsf{ch}}
\newcommand\Kum{\mathrm{Kum}}
\newcommand{\id}{\mathrm{id}}
\newcommand\Tan{\mathrm{Tan}}
\newcommand{\pt}{{\mathsf{p}}}

\title{Hilbert schemes of K3 surfaces, generalized Kummer, and   cobordism classes of hyper-K\"ahler manifolds}
\author{Georg Oberdieck\footnote{
The author is funded by the Deutsche Forschungsgemeinschaft (DFG) - OB 512/1-1.},
Jieao Song,
Claire Voisin\footnote{The author is supported by the ERC Synergy Grant HyperK (Grant agreement No. 854361).}}

\date{}

\newfont{\gothic}{eufb10}
\begin{document}
\maketitle
\setcounter{section}{-1}

\begin{abstract} We  prove that the complex  cobordism class of any hyper-K\"{a}hler manifold of dimension $2n$ is a unique combination with rational coefficients of classes of products of punctual Hilbert schemes of $K3$ surfaces. We also prove  a similar  result using   the generalized Kummer varieties instead of punctual Hilbert schemes. As a key step, we establish a closed formula for the top  Chern character of their tangent bundles. 
 \end{abstract}
\section{Introduction}
The cobordism ring denoted ${\rm MU}^*({\rm pt})$ in \cite{totaro} and $\Omega^*$ in \cite{milnor}  has the following easy description (which is not the original Milnor definition), see \cite{stong}. In degree $i$, consider the free abelian group $\mathcal{Z}^i$ generated by
$i$-dimensional compact manifolds $M$  equipped with a stable complex structure $\alpha$, namely  a complex vector bundle  structure on the real bundle  $T_M\oplus \mathbb{R}^k$, where $\mathbb{R}^k$ is the trivial real vector bundle of rank $k$ on $M$. It contains the subgroup $\mathcal{Z}^i_b$ generated by boundaries, namely, for any real $i+1$-fold $N$ with boundary  equipped with a stable complex structure $\alpha$, as
$T_{N\mid \partial N}\cong T_{\partial N}\oplus \mathbb{R}$, the stable complex structure on $N$ induces a stable complex structure on the boundary $\partial N$, defining the boundary $\partial(N,\alpha)$. The group ${\rm MU}^*({\rm pt})$  is then defined as the quotient $\mathcal{Z}^i/\mathcal{Z}^i_b$. The  ring structure comes from the addition given by the disjoint union, and the product is given by  the geometric product.
It is proved in \cite{milnor} that ${\rm MU}^*({\rm pt})$ is trivial in odd degree $*$ and torsion free in even degrees $*$. Furthermore it is also known that the cobordism class of a pair $(M,\alpha)$, with ${\rm dim}\,M=2i$ is determined by the Chern numbers
$$\int_M P_I(c_l(M,\alpha)),$$
where we use the orientation of $M$ defined by $\alpha$ to compute the integral,  the Chern classes $c_l(M,\alpha)$ are those of the complex vector bundle $T_M\oplus \mathbb{R}^k$ equipped with the stable complex structure $\alpha$, and the $P_I$ generate the space of degree $2i$ weighted homogeneous polynomials in the $c_j$ where ${\rm deg}\, c_j=2j$. We will in fact work with the $\mathbb{Q}$-vector space ${\rm MU}^*({\rm pt})\otimes \mathbb{Q}$ that we will denote ${\rm MU}^*({\rm pt})$ for convenience.

If we consider hyper-K\"{a}hler manifolds of dimension $2n$, or more generally compact complex $2n$-folds $X$ having an everywhere nondegenerate $(2,0)$-form $\sigma_X$ (not necessarily closed, not necessarily holomorphic), the existence of the isomorphism of complex
vector bundles
$$T_{X}^{1,0}\cong (T_{X}^{1,0})^*$$
given by $\sigma_X$ implies that $c_l(X)=0$ for $l$ odd. It follows that
the cobordism classes of such complex manifolds  are 
 determined by the Chern numbers
$$\int_X P(c_{2l}(X)),$$
where we use the complex orientation of $X$  to compute the integrals,   and the polynomials $P$ generate the space of degree $4n$ weighted homogeneous polynomials in the $c_{2l}$, where ${\rm deg}\, c_{2l}=2l$. These polynomials are generated by monomials $M_I$  indexed by partitions $I$ of $n$, namely
to a partition $I$ given by the decomposition $n=n_1+\ldots+n_k$, one associates the monomial
$$M_I=c_{2n_1}\ldots c_{2n_k}.$$

Starting with a $K3$ surface $S$, we can construct in each even dimension $2n$ the following set of  symplectic holomorphic manifolds, also indexed by partitions $I$ of $n$, namely, to a partition $I$ as above one associates
$$S^{[I]}:=S^{[n_1]}\times \ldots\times S^{[n_k]}.$$
Similarly, using the generalized Kummer varieties ${\rm Kum}_i(A)$ associated  with a $2$-dimensional complex torus or abelian surface (see \cite{beauville}) instead of the Hilbert schemes of $K3$ surfaces, we associate   to a partition $I$ as above  the symplectic holomorphic $2n$-fold
$${\rm Kum}_I(A):={\rm Kum}_{n_1}(A)\times \ldots {\rm Kum}_{n_k}(A).$$

The main result of this paper can be formulated as follows.

\begin{theo}\label{theogenerators} (a) The complex cobordism class of any compact complex manifold $X$ with trivial odd Chern classes is a unique combination with rational coefficients of classes  $S^{[I]}$, where $S$ is a $K3$ surface.

(b) The same result holds if one replaces the varieties $S^{[I]}$ by the varieties ${\rm Kum}_I(A)$.
\end{theo}

 In fact, the theorem that we will prove is even  more general, namely our results apply as well  to any compact complex manifold $X$ or complex cobordism class whose Chern numbers $\int_XM_I(c_1,\ldots, c_n)$, for any monomial $M_I$ involving nontrivially  an odd Chern class, are zero. For example, any complex fourfold $X$ with trivial first Chern class satisfies this property, while it can have $c_3(X)\not=0$. Similarly, complex $n$-folds with no nonzero odd degree  Chern classes in degree $\leq \frac{n}{2}$ satisfy this property. 
 More generally,  the rational subalgebra  $MU^*({\rm pt})_{\rm even}$ of $MU^*({\rm pt})$ consisting of cobordism classes with ``trivial  odd Chern numbers" in the  above sense is a free polynomial algebra over $\mathbb{Q}$ with one generator in each even dimension, and Theorem \ref{theogenerators} says that the cobordism classes of punctual Hilbert schemes of $K3$ surfaces, or of  the generalized Kummer varieties form a system of generators of this algebra.

 \begin{rema}\label{remarkpourneufpoints} {\rm It is known by \cite{EGL} that the cobordism class of  $S^{[I]}$ for a compact complex surface $S$ depends only on the Chern numbers $c_2(S),\,c_1(S)^2$. Hence we can replace in Theorem \ref{theogenerators} the $K3$ surface $S$ by any surface $S'$ with $c_1(S')^2=0$, $c_2(S')\not=0$, for example we can take for $S'$ the blow-up of $\mathbb{P}^2$ in $9$ points. }
 \end{rema}
 We will give a quick proof of Theorem \ref{theogenerators} (a)  in low dimension in Section \ref{secsmalldim}. In higher dimension, we will follow the following strategy, already  used by topologists.
The Milnor genus of a complex or almost complex manifold of complex dimension $m$ is defined as
\begin{eqnarray}
\label{eqmilnorgenus} M(X)=\int_X{\rm ch}_m(T_X),
\end{eqnarray}
where ${\rm ch}(T_X)=\sum_i{\rm ch}_i(T_X)$ is the Chern character of $X$ (see \cite{hirzebruch}).
As is classical in complex  cobordism theory (see \cite{milnor}, \cite{johnston}) and  will be recalled  in Section \ref{sec2}, Theorem \ref{theogenerators} is equivalent to  the following result concerning the Milnor genus  of $K3^{[n]}$ and ${\rm Kum}_n(A)$.
\begin{theo}\label{theononvanishing} (a) The Milnor genus $M(S^{[n]})$ is nonzero for all $n$.

(b) The Milnor genus $M({\rm Kum}_n(A))$ is nonzero for all $n$.
\end{theo}

Theorem \ref{theononvanishing} will be proved in Section \ref{secproof}, where an explicit formula for $M(S^{[n]})$ and $M({\rm Kum}_n(A))$ will be established (see Theorems \ref{thm:hilb case} and \ref{thm:Kum case}). In  Section \ref{sec2}, which is mostly  introductory, we will explain the equivalence between Theorems \ref{theogenerators} and \ref{theononvanishing}. In the last section of the paper, we will present a few natural questions left open by our results.

\vspace{0.5cm}

\noindent
{\bf Thanks.} {\it  We thank Olivier Debarre for interesting discussions that led us to this collaboration.}
\section{Theorem \ref{theogenerators}  in small dimension\label{secsmalldim}}
In complex dimensions $2n=2,\,4$ and $6$, the ring ${\rm MU}^{4n}({\rm pt})$ is very simple. Indeed, for $n=2$, the only class to integrate  is $c_2$. In dimension $4$, we get only $c_4$ and $c_2^2$. Finally, in dimension $3$, we get only
$c_2^3,\,c_2c_4,\,c_6$. In all three cases, the space has dimension $n$, which is not true anymore in higher dimensions (for example, in dimension $8$,  there is an extra monomial  $c_4^2$). In this case, there are natural Chern numbers on
$S^{[k]}$ that are enough to test the independence of the classes $S^{[k_i]}$ in ${\rm MU}^*({\rm pt})_{\rm even}$, namely the  $n$  numbers
$$\chi_k(X):=\chi(X,\Omega_X^k),$$
for $k\leq n$, $2n={\rm dim}\,X$.
We observe that, by Serre duality,   the other holomorphic Euler-Poincar\'{e} characteristics $\chi(X,\Omega_X^k)$ for $k>n$ do not bring further information.
By the Hirzebruch-Riemann-Roch formula, $\chi_k(X)$ is a polynomial of degree $2n$ in the Chern classes of $X$.
 There is nothing to say in dimension $2$. In dimensions $4$ and $6$, in order to prove Theorem \ref{theogenerators}, or equivalently Theorem \ref{theononvanishing}  by Corollary \ref{coroequiv},   it suffices to prove the following, for $S$ a $K3$ surface:

 \begin{prop} \label{propetite} (1) (${\rm dim}\,4$) The matrix $\begin{pmatrix}
\chi(\Omega_{S^{[2]}}) & \chi(\Omega_{S\times S}) \\
\chi(\Omega^2_{S^{[2]}}) & \chi(\Omega^2_{S\times S})
\end{pmatrix} $ has nonzero determinant.

(2) (${\rm dim}\,6$) The matrix
$$\begin{pmatrix}
\chi(\Omega_{S^{[3]}}) & \chi(\Omega_{S^{[2]}\times S})&\chi(\Omega_{S^{3}}) \\
\chi(\Omega^2_{S^{[3]}}) & \chi(\Omega^2_{S^{[2]}\times S})&\chi(\Omega^2_{S^{3}}) \\
\chi(\Omega^3_{S^{[3]}}) & \chi(\Omega^3_{S^{[2]}\times S})&\chi(\Omega^3_{S^{3}})
\end{pmatrix} $$ has nonzero determinant.
 \end{prop}
 \begin{proof} It is equivalent by Remark \ref{remarkpourneufpoints}, and  in fact easier to prove the same result for the surface $\Sigma$ obtained as the blow-up of  $\mathbb{P}^2$  in $9$ points. Indeed, in this case, the whole cohomology of the Hilbert scheme is of type $(p,p)$ and similarly for their products.  Thus we have
 $\chi(X,\Omega_X^k)=(-1)^kb_{2k}(X)$ for these varieties. The Betti numbers of $\Sigma^{[2]}$ and $\Sigma^{[3]}$ are computed by \cite{decami} or \cite{gottsche}. One has
 $$ b_2(\Sigma)=10,\,b_2(\Sigma^{[2]})=11  ,\,\,b_4(\Sigma^{[2]})= 66, $$
 $$b_2(\Sigma^{[3]})= 11 ,\,\,b_4(\Sigma^{[3]})= 77  ,\,\,b_6(\Sigma^{[3]})=342.$$
 By K\"{u}nneth decomposition, our matrices are thus $\begin{pmatrix}
-11 & -20 \\
66 & 102
\end{pmatrix} $ in case (1), and this matrix  has nonzero determinant and

$$\begin{pmatrix}
-11 & -21&-30 \\
77 & 177&303 \\
-342 & -682&-1060
\end{pmatrix} $$
in case (2), and this matrix  has nonzero determinant.
\end{proof}

\section{Reduction to Theorem \ref{theononvanishing}\label{sec2}}
The Chern character ${\rm ch}(E)$ of a complex vector bundle of rank $r$ on a topological space $X$ is defined
as
$${\rm ch}(E)=\sum_{i=1}^r{\rm exp}\,\lambda_i\in H^{2*}(X,\mathbb{Q}),$$
where the $\lambda_i$ are the formal roots of the Chern polynomial of $E$ (see \cite{hirzebruch}).
Its main properties are
\begin{eqnarray}\label{eqadd} {\rm ch}(E\oplus F)={\rm ch}(E)+{\rm ch}(F)
\end{eqnarray}
and, when $X$ is a manifold of real  dimension $k$,
\begin{eqnarray}\label{eqvanch}{\rm ch}_i(E)=0\,\,{\rm for}\,\,2i>k.
\end{eqnarray}
For a complex manifold $X$ we will use the notation  ${\rm ch}(X)={\rm ch}(T_X)$.
Let $X$ be a compact complex manifold of dimension $n$ which is a product
$$X\cong Y\times W$$
of complex manifolds of respective dimensions $n_Y,\,n_W<n$. Then,
as $T_X={\rm pr}_1^*T_Y\oplus {\rm pr}_2^* T_W$, where ${\rm pr}_i$ denotes the projection on the $i$-th factor, we get by (\ref{eqadd}) and (\ref{eqvanch})
\begin{eqnarray}\label{eqsumpourlesch}{\rm ch}_i(X)={\rm pr}_1^*{\rm ch}_i(Y) + {\rm pr}_2^*{\rm ch}_i(W),\end{eqnarray}
hence
$${\rm ch}_i(X)=0\,\,{\rm for}\,\, i>{\rm max}(n_Y,\,n_W).$$

The Milnor genus $M(X)$ defined in (\ref{eqmilnorgenus}) thus satisfies the following property
\begin{lemm} \label{lezeroprod} We have  $M(X)=0$ if $X$ is a product of two complex manifolds of  dimension smaller than $n$.
\end{lemm}

The formal properties above give the following criterion
\begin{prop}  \label{procriterion} Let $X_1,\ldots,\,X_i,\ldots,\,X_n$ be compact complex manifolds of dimension $2i$ with vanishing odd Chern classes : $c_{2l+1}(X_i)=0$. Then,  $\lambda_i := M(X_i)$ is nonzero for any $i$, if and only if  any complex cobordism class  of even dimension $\leq 2n$ with vanishing ``odd Chern numbers'' can be written uniquely as a rational combination of products
$$X_I:=X_{i_1}\times\ldots\times X_{i_k},\,\,\sum_l i_l\leq n.$$
\end{prop}
\begin{proof} The ``only if" follows from Lemma \ref{lezeroprod} which says that ${\rm ch}_{2i}$ can have a nonzero integral on $X_{i_1}\times\ldots\times X_{i_k}$ only  for $I=\{i\}$, that is, when  $X_{i_1}\times\ldots\times X_{i_k}=X_i$.

 We have to  prove   that the products $X_{i_1}\times\ldots\times X_{i_k}$ form a basis over $\mathbb{Q}$ of the subring ${\rm MU}^*({\rm pt})_{\rm even}$ of the cobordism ring of classes $\alpha$ with vanishing  ``odd Chern numbers" $\int_\alpha M_I(c_i)$, where the monomial $M_I$ involves an odd Chern class. Equivalently, we have to show that for any such  class $\alpha\in {\rm MU}^{4n}({\rm pt})_{\rm even}$, there are unique  rational coefficients $\alpha_{I}$ indexed by partitions of $n$, such that
$$\int_\alpha P(c_2,\ldots,c_{2n})=\sum_I\alpha_I\int_{X_\cdot^I}P(c_2(X_I),\ldots,c_{2n}(X_I))$$
for any degree $2n$ weighted polynomial $P$ in the variables $c_{2l}$.
Instead of using the Chern classes $c_{2i}$ as generators, we can use the Chern characters classes ${\rm ch}_{2i}$ which are related to the Chern classes by the Newton formulas. 
We argue by induction on the dimension and conclude that for any
$i<n$, there exists a combination
\begin{eqnarray}\label{eqYcorr} Y_i=X_i+\sum_{I,l(I)\geq 2}\alpha_I X_I \in {\rm MU}^{4i}({\rm pt}),\end{eqnarray}
where, in the above sum, $I$ runs through the partitions $i=\sum_{l=1}^k i_l$ of $i$ and $l(I):=k$,
with the following property : for any degree $2i$ monomial
$M_K={\rm ch}_{2}^{k_2}\ldots {\rm ch}_{2i}^{k_{2i}}$ in the Chern characters ${\rm ch}_l$ with $l$ even, one has
\begin{eqnarray}\label{eqcorrvan} M_K(Y_i)=M_K(X_i)+\sum_{I, l(I)\geq 2}\alpha_I\int_{X_I} M_K({\rm ch}_2(X_I),\ldots,{\rm ch}_{2i}(X_I))=0 \,\,{\rm if}\,\,M_K\not={\rm ch}_{2i}.
\end{eqnarray}
Furthermore, equation (\ref{eqYcorr}), Lemma \ref{lezeroprod} and our assumptions show that
$M(Y_i)=\lambda_i\not=0$.
Formulas (\ref{eqsumpourlesch}) and (\ref{eqcorrvan}) then imply that
for any  product $Y_J=\prod_{j_1+\ldots +j_k=i} Y_{j_l}$ with $l\geq 2$ (hence all $j_s$ smaller than $i$), and any monomial $M_K$ as above of weighted degree $2i$, one has
$M_K(Y_J)=0$ for $K\not=J$, $M_K(Y_K)\not=0$. Finally, we have by assumption ${\rm ch}_{2i}(X_i)\not=0$, so
 $X_i$ and the $Y_J$ for the partitions $J$ of $i$ such that  $l(J)\geq 2$ form a basis of ${\rm MU}^{4i}({\rm pt})_{\rm even}$.
\end{proof}
\begin{rema}\label{remaprodproj}  {\rm The same criterion (without assumption on the odd Chern classes) was used by topologists  to prove that the complex  cobordism ring with rational coefficients is generated in degree $n$ by products of projective spaces $\mathbb{P}^{i_k}$ with $\sum_ki_k=n$. It suffices to prove that $M(\mathbb{P}^r)\not=0$, which is quite easy using the Euler exact sequence which gives
$${\rm ch}(\mathbb{P}^r)=(r+1){\rm exp}(h)-1,$$ with $h=c_1(\mathcal{O}_{\mathbb{P}^r}(1))$. }
\end{rema}
We now  get in particular
\begin{coro}\label{coroequiv}  Theorem \ref{theononvanishing} is equivalent to  Theorem \ref{theogenerators}.
\end{coro}

\section{Proof of Theorem \ref{theononvanishing} \label{secproof}}
The proof of Theorem~\ref{theononvanishing}
will use the description of the cohomology of Hilbert schemes of points of surfaces in terms of Nakajima operators.
In particular, we will use a result of Li, Qin and Wang \cite{LQW} which
for $K$-trivial surfaces expresses the operator of multiplication by tautological classes in terms of the Nakajima basis.
We refer to \cite{NOY} for an overview of the main definitions in the subject, and for the conventions that we follow.

We will prove the following closed evaluations,
which imply Theorem~\ref{theononvanishing}.

\begin{thm} \label{thm:hilb case}
For any surface $S$ with $c_1(S) = 0$ in $H^{2}(S,\BQ)$, we have for all $n \geq 1$:
\[ \int_{S^{[n]}} \ch_{2n}( T_{S^{[n]}} ) = (-1)^n \frac{c_2(S)}{24} \frac{(2n+2)!}{n!^4 (2n-1)}. \]
\end{thm}

\begin{thm} \label{thm:Kum case} For any abelian surface $A$, we have for all $n \geq 1$:
\begin{align*}
\int_{\Kum_{n}(A)} \ch_{2n}( T_{\Kum_n(A)}) & = (-1)^n\frac{(2n+2)!}{n!^4}.
\end{align*}
\end{thm}

\subsection{Combinatorial identities}
\begin{lemma} \label{lemma}
For $k, n\in\BN$, we have the following identities:
\begin{enumerate}
\item[(1)] $\sum_{i=0}^n\binom{n}{i}^2=\binom{2n}n$;
\item[(2)] $\sum_{i=0}^ni\binom{n}{i}^2=\frac{n}2\binom{2n}n$;
\item[(3)] $\sum_{i=0}^ni^2\binom{n}{i}^2=\frac{n^3}{2(2n-1)}\binom{2n}n$;
\item[(4)] $\sum_{i=0}^{k}(-1)^i\binom{n}{i} = (-1)^{k}\binom{n-1}{k}$;
\item[(5)] $\sum_{i=0}^{k}(-1)^i i\binom{n}{i} = (-1)^{k}n \binom{n-2}{k-1}$.
\end{enumerate}
\end{lemma}
\begin{proof}
For (1), one can compare the degree-$n$ coefficient of the polynomial
$(1+x)^{2n}$: the left hand side is obtained using the identity
$(1+x)^{2n}=(1+x)^n(1+x)^n$, while the right hand side is simply the binomial coefficient.
For (2) and (3), we consider the polynomials $(1+x)^n\cdot
\frac{d}{d x}(1+x)^n$ and $(1+x)^n\cdot
\big(\frac{d}{d x}\big)^2(1+x)^n$, and follow the same idea as
(1).

For (4), we consider the degree-$k$ coefficient of the polynomial
$(1-x)^{n-1}$: the right hand side is again just the binomial coefficient, while the left hand side is
obtained using the Taylor expansion
$(1-x)^{n-1}=\frac1{1-x}\cdot (1-x)^n=(1+x+x^2+\cdots)\cdot (1-x)^n$.
Similarly, for (5) we consider
$-n(1-x)^{n-2}=\frac1{1-x}\cdot\frac{d}{d x}(1-x)^n$.
\end{proof}
\begin{prop} \label{prop_combinatorics}
We have the following identity
\[
\sum_{l=0}^{n}\sum_{m=0}^{l-1}(-1)^{m+l+1} \frac{l-m} { m!\, l!\, (n-m)!\, (n-l)! }
=\frac{n}{2(2n-1)}\frac{(2n)!}{n!^4}.
\]
\end{prop}
\begin{proof}
We rewrite the left hand side using the combinatorial identities from Lemma~\ref{lemma}
\begin{align*}
& \sum_{l=0}^{n}\sum_{m=0}^{l-1}
(-1)^{m+l+1} \frac{l-m} { m! l! (n-m)! (n-l)! } \\
& =
\frac1{n!^2}\sum_{l=0}^n\sum_{m=0}^{l-1}(-1)^{m+l+1}(l-m)\binom{n}{m}\binom{n}{l}\\
(\text{take out $l$})&=
\frac1{n!^2}\sum_{l=0}^n(-1)^l\binom{n}{l}\left(l\sum_{m=0}^{l-1}(-1)^{m+1}\binom{n}{m}+\sum_{m=0}^{l-1}(-1)^m m\binom{n}{m}\right)\\
\left(\substack{\text{using Lemma~\ref{lemma}}\\\text{(4) and (5)}}\right)&=
\frac1{n!^2}\sum_{l=0}^n(-1)^l\binom{n}{l}\left((-1)^l
l\binom{n-1}{l-1}+(-1)^{l-1}n\binom{n-2}{l-2}\right)\\
&=
\frac1{n!^2}\sum_{l=0}^n\binom{n}{l}\left(
\frac{l^2}n\binom{n}{l}-\frac{l^2-l}{n-1}\binom{n}{l}\right)\\
&=
\frac1{n!^2}\left(\frac1{n-1}\sum_{l=0}^nl\binom{n}{l}^2-\frac1{n(n-1)}\sum_{l=0}^nl^2\binom{n}{l}^2\right)\\
\left(\substack{\text{using Lemma~\ref{lemma}}\\\text{(2) and (3)}}\right)&=
\frac{n}{2(2n-1)}\frac{(2n)!}{n!^4}.
\end{align*}
\end{proof}

\subsection{Hilbert schemes of points}
Let $\CZ \subset S^{[n]} \times S$ be the universal subscheme and let
$\pi, \pi_S$ be the projections of $S^{[n]} \times S$ to the factors.
For any $\gamma \in H^{\ast}(S)$ and $d \in \BZ$ let
\[ \fG_d(\gamma) : H^{\ast}(S^{[n]}) \to H^{\ast}(S^{[n]}) \]
be the operator of multiplication with the class
$\pi_{\ast}( \ch_d( \CO_\CZ - \CO_{S^{[n]} \times S}) \cdot \pi_S^{\ast}(\gamma) )$.

Let from now on $S$ be a surface with $c_1(S) = 0$ in $H^2(S,\BQ)$. Then by a result of Li, Qin and Wang \cite[Thm.4.6]{LQW} we have that
\begin{equation} \label{Gdexpansion}
\fG_d(\gamma) = - \sum_{ \substack{ |\lambda| = 0 \\ \ell(\lambda) = d}} \frac{\Fq_{\lambda}}{\lambda!}(\Delta_{\ast}(\gamma)) + \sum_{ \substack{ |\lambda|=0 \\ \ell(\lambda) = d-2 }} \frac{s(\lambda)}{24 \cdot \lambda!} \Fq_{\lambda}( \Delta_{\ast}( \gamma \cdot e(S) ) )
\end{equation}
where $e(S) \in H^4(S)$ is the Euler class of $S$ and $\Fq_m(\alpha)$ are the Nakajima Heisenberg operators;
the other notations here follow \cite[Section 4]{NOY}.\footnote{
There is one exception: our definition for $\fG_d(\gamma)$ agrees with \cite{NOY} in case $d \geq 1$,
while for $d=0$ we have $\fG_0(\gamma) = - \left( \int_S \gamma \right) \id$ (instead of $\fG_0(\gamma)=0$ in \cite{NOY}). The advantage is that \eqref{Gdexpansion} holds now for all $d \in \BZ$.}

The tangent bundle of the Hilbert scheme can be expressed as an relative Ext sheaf of the universal ideal sheaves  \cite[Prop 2.2]{EGL}.
This gives an expression for the operator of multiplication with $\ch_k(T)$ in terms of the $\fG$'s 
as follows
\begin{align} \mathrm{mult}_{\ch_k(T_{S^{[n]}})} 
& = \sum_{i+j=k+2} (-1)^{j+1} \fG_i \fG_j(\Delta) + \frac{c_2(S)}{12} \sum_{i+j=k} (-1)^{j+1} \fG_i(\pt) \fG_j(\pt)
\label{chk}
\end{align}
where $k \geq 1$ and $\pt \in H^4(S)$ is the class of a point on $S$;
see also \cite[4.9]{NOY}.
Hence Theorem~\ref{thm:hilb case} is implied by the following two lemmas:

\begin{lemma} \label{lemma1}
\[ \sum_{i+j=2n} (-1)^{j+1}  \int_{S^{[n]}} \fG_i(\pt) \fG_j(\pt) 1_{S^{[n]}} = (-1)^{n+1} \frac{(2n)!}{n!^4} \]
\end{lemma}
\begin{proof}
In the Nakajima basis the unit of $H^{\ast}(S^{[n]})$ is $\frac{1}{n!} \Fq_1(1)^n 1_{S^{[0]}}$ 
where we let $1_{S^{[0]}}$ denote the unit in the cohomology of $S^{[0]} = \{ \ast \}$ (the subscript $S^{[0]}$ is usually dropped in what follows).
We hence have to evaluate
\begin{align}
& \sum_{i+j=2n} (-1)^{j+1} \int_{S^{[n]}} \fG_i(\pt) \fG_j(\pt) \frac{1}{n!} \Fq_1(1)^n 1 \notag \\
& = \sum_{i+j=2n} (-1)^{j+1} \int_{S^{[n]}} \left( \sum_{l(\lambda) = i, |\lambda|=0} \frac{\Fq_{\lambda}( \Delta_{\ast}(\pt))}{\lambda!} \right)
\left( \sum_{l(\widetilde{\lambda}) = j, |\widetilde{\lambda}|=0} \frac{\Fq_{\widetilde{\lambda}}( \Delta_{\ast}(\pt))}{\widetilde{\lambda}!} \right) \frac{\Fq_1(1)^n}{n!} 1 \label{abcs}
\end{align}
The (complex\footnote{The complex degree $\deg_{\BC}(\gamma)$ is half the real degree, i.e. $\gamma \in H^{2 \deg_{\BC}(\gamma)}$.}) cohomological degree of a Nakajima cycle $\Fq_{k_1}(\gamma_1) \cdots \Fq_{k_r}(\gamma_r) 1$ lying in $H^{\ast}(S^{[n]})$ is $n - r + \sum_i \deg_{\BC}(\gamma_i)$.
Hence for the integral of such a cycle to be non-zero, we need $k_i = 1$ and $\deg_{\BC}(\gamma_i) = 2$ for all $i$.
In particular, the term $\Fq_1(1)^n$ appearing in \eqref{abcs} has to be transformed into a multiple of $\Fq_1(\pt)^n$ under the operators $\fG_i(\pt) \fG_j(\pt)$.
Hence among the $\Fq_{\lambda}$ and $\Fq_{\tilde{\lambda}}$ we must have $n$ operators of the form $\Fq_{-1}$ and $n$ operators $\Fq_1$. Since this accounts for all possible Nakajima operators which can appear, we need that
$\lambda = (-1)^a (1)^a$ and $\tilde{\lambda} = (-1)^b (1)^b$ where $i=2a$ and $j=2b$.
The above expression thus evaluates to
\begin{align*}
& = (-1) \sum_{a+b = n} \frac{1}{a!^2 b!^2} \int_{S^{[n]}} \Fq_1(\pt)^a \Fq_{-1}(\pt)^a \Fq_1(\pt)^b \Fq_{-1}(\pt)^b \frac{ \Fq_1(1)^n}{n!} 1 \\
& = (-1) \sum_{a+b = n} \frac{1}{a!^2 b!^2} (-1)^n \\
& = (-1)^{n+1} \frac{(2n)!}{n!^4}
\end{align*}
where in the last equality we used the first part of Lemma~\ref{lemma}.
\end{proof}

\begin{lemma}
\[ \sum_{i+j=2n+2} (-1)^{j+1} \int_{S^{[n]}} \fG_i \fG_j(\Delta) 1_{S^{[n]}}  = c_2(S) (-1)^n \frac{(2n)!}{n!^4} \left[ \frac{n}{12} + \frac{n}{2 (2n-1)} \right] \]
\end{lemma}
\begin{proof}
We insert the expansion \eqref{Gdexpansion} for $\fG_i$. The contribution from the second term in \eqref{Gdexpansion}
can be computed by the same methods which were used in Lemma~\ref{lemma1}. The result is $e(S)/24 \sum_{a+b=n}(-1)^n 2a / (a!^2 b!^2)$.
The same applies to the contribution from the second term in $\fG_j$.
Inserting this and using part (2) of Lemma~\ref{lemma} we find that:
\[ 
\sum_{i+j=2n+2} (-1)^{j+1} \int_{S^{[n]}} \fG_i \fG_j(\Delta) 1_{S^{[n]}}  = 
I + c_2(S) (-1)^n \frac{n^2 (2n-1)!}{6 \cdot n!^4} \]
where $I$ is the contribution from the first terms in $\fG_i$ and $\fG_j$,
that is
\[
I = \sum_{i+j=2n+2} (-1)^{j+1} \int_{S^{[n]}} 
\left( \sum_{l(\lambda) = i, |\lambda|=0} \frac{\Fq_{\lambda}( \Delta_{\ast}(\Delta_1))}{\lambda!} \right)
\left( \sum_{l(\widetilde{\lambda}) = j, |\widetilde{\lambda}|=0} \frac{\Fq_{\widetilde{\lambda}}( \Delta_{\ast}(\Delta_2))}{\widetilde{\lambda}!} \right) \frac{\Fq_1(1)^n}{n!} 1
\]
where $\Delta_1, \Delta_2$ stands for summing over the K\"unneth factors of
the diagonal in $H^{\ast}(S \times S)$.
With similar reasoning as before (i.e. among the $\Fq_{\lambda}$ and $\Fq_{\tilde{\lambda}}$ we need $n$ operators $\Fq_{1}$ and $\Fq_{-1}$ each) we now compute:
\begin{align*}
I 
& = \sum_{\ell=1}^{n} \sum_{m=0}^{\ell - 1} \int_{S^{[n]}} \frac{  \Fq_{1}^{n-m} \Fq_{-1}^{n - \ell} \Fq_{-(\ell-m)} \Fq_{\ell-m} \Fq_1^m \Fq_{-1}^{\ell}(\Delta) }{b_{m,\ell}} (-1)^{m + \ell} 
\frac{ \Fq_1(1)^n }{n!} 1
\end{align*}
with
\[ b_{m,\ell} = \begin{cases}
m! \ell! (n-m)! (n-\ell)! & \text{ if } m < \ell - 1 \\
\ell !^2 (n- \ell + 1)!^2 & \text{ if } m = \ell - 1.
\end{cases}
\]
Commuting the negative Nakajima operators to the right
and using the Nakajima commutation relations for cases $m=\ell-1$ and $m<\ell-1$ separately, we get
\[ I = 
e(S) (-1)^n \sum_{\ell=1}^{n} \sum_{m=0}^{\ell-1} (-1)^{m+\ell+1} \frac{( \ell - m)}{m! \ell! (n-m)! (n-\ell)!}
= \frac{ e(S) (-1)^n n}{2(2n-1)}\frac{(2n)!}{n!^4}
\]
where we applied Proposition~\ref{prop_combinatorics} in the last step.
\end{proof}

\subsection{Generalized Kummer varieties}
We first compute the class of $\Kum_n(A)$ in the Nakajima basis of $A^{[n+1]}$.
\begin{lemma} \label{lemma:Kum class} In $H^{4}(A^{[n+1]})$ we have
\[ [ \Kum_n(A) ] = \fG_2(\alpha) \fG_2(\beta) \fG_2(\gamma) \fG_2(\delta) 1_{A^{[n+1]}} \]
for any $\alpha, \beta, \gamma, \delta \in H^1(A)$ such that $\int_{A} \alpha \beta \gamma \delta = 1$.
\end{lemma}
\begin{proof}
Let $\sigma : A^{[n+1]} \to A$ be the sum map. We have
\[ [ \Kum_n(A) ] = \sigma^{\ast}(\pt). \]
Hence it suffices to show that $\sigma^{\ast}(\alpha) = G_2(\alpha)$ for any $\alpha \in H^1(A)$,
where we let $G_2(\alpha) = \fG_2(\alpha) 1_{A^{[n+1]}}$.
Consider $x \in H^3(A,\BZ) = H_1(A,\BZ)$ and let $L(x) = \Fq_1(x) \Fq_{1}(\pt)^{n} 1$.
When $x$ is represented by a singular chain, than $L(x)$ is represented by the chain obtained from the former by adding $n-1$ distinct points to it.
This shows that $\sigma_{\ast} L(x) = x$,
and hence 
\[ \int_{S^{[n]}} \sigma^{\ast}(\alpha) \cdot L(x) = \int_A \alpha \cdot \sigma_{\ast} L(x) = \int_A \alpha x. \]
On the other hand, a direct calculation using the Nakajima operators also shows $\int G_2(\alpha) \cdot L(x) = \int_{A} \alpha x$.
Since the $L(x)$ generate $H_1(A^{[n+1]})$ this yields the claim.
\end{proof}

Since $[\fG_2(x), \Fq_1(y)] = \Fq_1(xy)$ for all $x, y \in H^{\ast}(S)$ one finds that
\begin{equation} \label{abss}
[\Kum_n(A)]
=
\sum_{\pi = \{ \pi_i \}} \sigma_{\pi} \frac{1}{ (n+1-\ell(\pi))!} \prod_i \Fq_1\left( \prod_{x \in \pi_i} x \right) \Fq_1(1)^{n+1-\ell(\pi)} 1
\end{equation}
with the following notation: 
\begin{itemize}[itemsep=0pt]
\item $\pi$ runs over all set partitions of $\{ \alpha, \beta, \gamma, \delta \}$ with $l(\pi)$ parts,
\item $\sigma_{\pi} \in \{ \pm 1 \}$ is the sign obtained from bringing
$\prod_i \prod_{x \in \pi_i} x$ into the order $\alpha \beta \gamma \delta$,
\item in case $n \leq 2$ we sum only over set partitions with $l(\pi) \leq n+1$. 
\end{itemize}
The first terms read:
\begin{multline*}
[\Kum_n(A)]
= \frac{1}{n!} \Fq_1( \pt ) \Fq_1(1)^{n} 1 + \frac{1}{(n-1)!} \Fq_1(\alpha) \Fq_1(\beta \gamma \delta) \Fq_1(1)^{n-1} 1 \\
+ \ldots + \frac{1}{(n-3)!}
\Fq_1(\alpha) \Fq_1(\beta) \Fq_1(\gamma) \Fq_1(\delta) \Fq_1(1)^{n-3} 1.
\end{multline*}

\begin{lemma} \label{lemma:Delta Kum}
\[
\int_{A^{[n+1]}} \Fq_1^{n+1} \Fq_{-1}^{n+1}(\Delta)
[\Kum_n(A)]
= (n+1)^4
\]
\end{lemma}
\begin{proof}
Using Lemma~\ref{lemma:Kum class}, equation \eqref{abss} and the straightforward evaluation
\[
\sigma_{\pi} \int_{A^{[n+1]}} \Fq_1^{n+1} \Fq_{-1}^{n+1}(\Delta)
\prod_i \Fq_1\left( \prod_{x \in \pi_i} x \right) \Fq_1(1)^{n+1-\ell(\pi)} 1
=
(-1)^{n+1} (n+1)!
\]
for every $\pi$, we find that
\begin{align*}
\int_{A^{[n+1]}} \Fq_1^{n+1} \Fq_{-1}^{n+1}(\Delta)
[\Kum_n(A)]
& = \sum_{\pi} \frac{ (n+1)! (-1)^{n+1} }{ (n+1 - l(\pi))!} \\
& = (-1)^{n+1} (n+1) \Big[ 1 + 7 n + 6 n (n-1) + n (n-1) (n-2) \Big] \\
& = (-1)^{n+1} (n+1)^4
\end{align*}
\end{proof}

\begin{proof}[Proof of Theorem~\ref{thm:Kum case}]
We have the exact sequence
\[ 0 \to T_{\Kum_n(A)} \to T_{A^{[n+1]}}|_{\Kum_n(A)} \to \sigma^{\ast}(T_A)|_{\Kum_n(A)} \to 0 \]
which together with \eqref{Gdexpansion} and \eqref{chk} (using $e(A) = 0$) shows that
\begin{align*}
\int_{\Kum_n(A)} \ch_{2n}(T_{\Kum_n(A)})
& = \int_{A^{[n+1]}} \ch_{2n}(T_{A^{[n+1]}}) \cap [\Kum_n(A)] \\
& = \sum_{i+j=2n+2} (-1)^{j+1} \int_{A^{[n+1]}} \fG_i \fG_j(\Delta) 
[\Kum_n(A)]
\end{align*}

Consider the expansion $\fG_i \fG_j(\Delta) = \sum_{l(\lambda) = i, l(\tilde{\lambda}) = j} \Fq_{\lambda} \Fq_{\tilde{\lambda}}(\Delta) / (\lambda! \tilde{\lambda}!)$.
Since $\Fq_{\tilde{\lambda}}$ acts on \eqref{abss} which consists only of terms of the form
$\prod_i \Fq_1(x_i) \, 1$,
for a summand to contribute, $\tilde{\lambda}$ can only have negative parts equal to $-1$. Assume $\tilde{\lambda}$ has a positive part $k>1$. Then $\lambda$ has to have a corresponding negative part $-k$, and these two parts have to interact when commuting all negative Nakajima operators to the right.
However, this will yield the term
\[ [\Fq_{-k}, \Fq_k] \Fq_{\lambda'} \Fq_{\tilde{\lambda}'}( \pi_{12 \ast}( \Delta_{12} \Delta_{1 2 \cdots (l(\lambda) + l(\tilde{\lambda}))}) ) = 
-k \Fq_{\lambda'} \Fq_{\tilde{\lambda}'}( e(A) \Delta ) = 0 \]
where $\pi_{12}$ is the projection away from the first two factors and $\lambda', \tilde{\lambda}'$ are the partitions $\lambda, \tilde{\lambda}$ without the parts $k, -k$.
We conclude that only the summands with $\lambda = (-1)^a (1)^a$ and $\tilde{\lambda} = (-1)^b (1)^b$ where $i=2a$ and $j=2b$ can contribute to the integral. Moreover, applying a similar argument we have
$\Fq_1^a \Fq_{-1}^a \Fq_{1}^b \Fq_{-1}^b (\Delta) = \Fq_{1}^{a+b} \Fq_{-1}^{a+b}(\Delta)$.

We thus find the following expression:
\begin{align*}
& =
\sum_{a+b=n+1} \frac{(-1)}{a!^2 b!^2} 
\int_{A^{[n+1]}} \Fq_1^{n+1} \Fq_{-1}^{n+1}(\Delta) [\Kum_n(A)] \\
& = \sum_{a+b=n+1} \frac{(-1)}{a!^2 b!^2} (-1)^{n+1} (n+1)^4 \\
& = (-1)^{n} \frac{(2n+2)!}{n!^4}.
\end{align*}
where we used the first part of Lemma~\ref{lemma}.
\end{proof}

The computations above can be generalized to arbitrary products of Chern characters.
The following qualitative result is almost immediate:
\begin{prop} \label{prop:positivity}
Let $n \geq 1$. For any partition $n = k_1 + k_2 + \ldots + k_r$ we have
\[
(-1)^n \int_{\Kum_n(A)} \ch_{2 k_1}(T_{\Kum_n(A)}) \cdots \ch_{2 k_r}(T_{\Kum_n(A)}) > 0.
\]
\end{prop}
\begin{proof}
Let $n-1 = k_1 + \ldots + k_{r}$ be a partition of $n-1$. Then
\begin{multline*}
\int_{\Kum_{n-1}( A)} \ch_{2 k_1}(T) \cdots \ch_{2 k_{r}}(T) \\
=
\sum_{\substack{i_1 + j_1 = 2k_1+2 \\ \ldots \\ i_r + j_r = 2 k_r+2 }}
(-1)^{j_1 + \ldots + j_r + r}
\int_{A^{[n]}} \fG_{i_1} \fG_{j_1}(\Delta) \cdots \fG_{i_r} \fG_{j_r}(\Delta) [\Kum_{n-1}(A)]
\end{multline*}
We express the $\fG_d$ in terms of Nakajima operators via \eqref{Gdexpansion},
which produces a sum consisting of summands with precisely
\[ \sum_{s=1}^{r} i_s + j_s = \sum_s (2 k_s + 2) = 2n + 2 (r-1) \]
Nakajima factors acting on the class of $\Kum_{n-1}(A)$.
When commuting all negative Nakajima operators to the right,
we see that for a term to contribute
there have to be at least $r-1$ Nakajima interactions between these $2n+2(r-1)$ factors.
Moreover, since $e(A) = 0$ (compare the proof of Theorem~\ref{thm:Kum case})
only the following is allowed:
\begin{itemize}[itemsep=0pt]
\item[(a)] There can be no Nakajima interacts between factors belonging to the same $\fG_{i_s} \fG_{j_s}(\Delta)$.
\item[(b)] There can be at most one Nakajima interaction between factors belonging to $\fG_{i_s} \fG_{j_s}(\Delta)$ and $\fG_{i_{s'}} \fG_{j_{s'}}(\Delta)$ for $s \neq s'$.
\end{itemize}
This shows that there can be at most $r-1$ Nakajima interactions.
The total sign contribution from these Nakajima interactions is $(-1)^{r-1}$
and the outcome will be a multiple of the operator $\Fq_{1}^n \Fq_{-1}^n(\Delta)$.
By Lemma~\ref{lemma:Delta Kum} the 
degree of $\Fq_{1}^n \Fq_{-1}^n(\Delta) [\Kum_{n-1}(A)]$ yields a sign of $(-1)^n$.
Since there always is at least one summand that contributes with a non-zero values,
the claim now follows as soon as we can prove that
$j_1 + \ldots + j_r$ is even.

If $\lambda = \big( \ldots (-2)^{l_2} (-1)^{l_1} (1)^{l_1} (2)^{l_2} \ldots \big)$ is a generalized partition of size $|\lambda| = \sum_{i} i l_i = 0$, then
by considering this equality mod $2$ we get that the number of odd parts $l_{\text{odd}} := \sum_j l_{2j+1}$ is even, and hence that $l(\lambda)$ is equal to the number of even parts $l_{\text{even}}(\lambda) := \sum_{j} l_{2j}$ modulo $2$.
Let $\lambda_{s}, \tilde{\lambda}_{s}$ be the generalized partitions appearing in a given summand of $\fG_{i_s} \fG_{j_s}$. We see
\[ (-1)^{j_1 + \ldots + j_r} = (-1)^{l_{\text{even}}(\lambda_1) + \ldots l_{\text{even}}(\lambda_r)}. \]

Moreover, since $i_s+j_s$ is even, for every $s$ we have $l_{\text{even}}(\lambda_s) + l_{\text{even}}(\tilde{\lambda}_s)$ is even.
This shows that there is always an even number of even Nakajima factors in $\fG_{i_s} \fG_{j_s}(\Delta)$.
Let $m$ be the number of $s \in \{ 1, \ldots, r \}$ such that there exists even Nakajima factors in $\fG_{i_s} \fG_{j_s}(\Delta)$. Since all even Nakajima factors have to interact with each other,
we see that there are at least $2m$ Nakajima interactions between these $m$ factors,
This implies that either (a) or (b) above is violated, and the corresponding contribution vanishes.
Hence for any non-zero summand contributing to the Chern character number,
all Nakajima factors are odd, so we have $j_s \equiv 0 ( 2 )$
and therefore $(-1)^{j_1 + \ldots + j_r}$ even.
\end{proof}

\begin{rmk}
Arbitrary Chern character numbers of $\Kum_n(A)$ can be computed in a parallel manner,
however the expressions become more complicated.
For example, the double Chern character numbers of the generalized Kummer for $0<k<n$ are given as
\[
\begin{aligned}
&\int_{ \Kum_n(A)}
\ch_{2k} \ch_{2n-2k}\\
=&4(-1)^n(n+1)^4(2k+1)!(2n-2k+1)!\sum_{i=0}^k\frac{2i+1}{((k-i)!(k+i+1)!(n-k-i)!(n-k+i+1)!)^2}.
\end{aligned}
\]
where $\ch_k = \ch_k(T_{\Kum_n(A)})$ are the Chern characters of the tangent bundle.
For $k = 1$ one gets
\[
\int_{ \Kum_n(A)}
\ch_2 \ch_{2n-2}
=
(-1)^n \frac{(2n)!}{n!^4} \left( 4 n (n+1)^2 (n^2 + n + 1) \right).
\]
\end{rmk}

\section{Remarks and open questions\label{secremaopen}}
A first obvious question is the following
\begin{question} Compute  $M(\Sigma^{[n]})$ for any smooth projective surface $\Sigma$.
\end{question}
More precisely, it is a consequence of \cite{EGL} that we have a formula
\begin{eqnarray}\label{eqexpressiongenerale} M(\Sigma^{[n]})=\alpha_n c_1(\Sigma)^2+\beta_n c_2(\Sigma),
\end{eqnarray}
so the question is to compute $\alpha_n$ and $\beta_n$. 
Formula (\ref{eqexpressiongenerale}) follows from the main result of \cite{EGL} which says that $M(\Sigma^{[n]})$ depends only on $c_1(\Sigma)^2$ and $c_2(\Sigma)$  and from  the formula
$$ \Sigma^{[n]}=\sqcup_{k+l=n} \Sigma_1^{[k]}\times \Sigma_2^{[l]}$$
when $\Sigma=\Sigma_1\sqcup \Sigma_2$, which by Lemma \ref{lezeroprod} gives
$$M(\Sigma^{[n]})=M(\Sigma_1^{[n]})+ M(\Sigma_2^{[n]}),$$ proving that
$M(\Sigma^{[n]})$ is a linear function of $c_1(\Sigma)^2$ and $c_2(\Sigma)$. Equation (\ref{eqexpressiongenerale}) suggests that another approach to Theorem \ref{theononvanishing} would be by computing the Milnor genus of $\Sigma^{[n]}$ for two conveniently chosen surfaces $\Sigma$, in the spirit of \cite{voisinsegre}. 

Theorem~\ref{thm:hilb case} shows that $\beta_n=(-1)^n \frac{(2n+2)!}{24(2n-1)(n!)^4}$. The Milnor genus of $(\mathbb P^2)^{[n]}$ can be numerically computed using Bott's residue formula for small values of $n$, so we get the following list of $\alpha_n$.
\begin{center}
\begin{tabular}{ |c|c| } 
\hline
$n$ & $\alpha_n$ \\
\hline
1 & $1/2$\\
2 & $-5/12$\\
3 & $91/540$\\
4 & $-67/1680$\\
5 & $5599/907200$\\
6 & $-8047/11975040$\\
7 & $295381/5448643200$\\
8 & $-17616097/5230697472000$\\
9 & $797006281/4801780279296000$\\
10 & $-404188861/60822550204416000$\\
11 & $15479922001/70250045486100480000$\\
12 & $-8942373821/1454175941562279936000$\\
\hline
\end{tabular}
\end{center}

Turning to  hyper-K\"{a}hler geometry, an obvious open question, that was our original motivation for formulating Theorem \ref{theogenerators}, is
\begin{question}\label{question1} What are the constraints on the complex cobordism classes of hyper-K\"{a}hler manifolds?
\end{question}
In view of Theorem \ref{theogenerators}, we can rephrase this question in terms of inequalities or equalities between the coefficients $\alpha_I(X)$ (resp. $\beta_I$)  given by  Theorem \ref{theogenerators}, expressing the class of $X$ as a combination of classes of the $S^{[I]}$ (resp. ${\rm Kum}_I(A)$). One obvious restriction is the affine relation given by the fact that $\chi(X,\mathcal{O}_X)=n+1$ for $X$ hyper-K\"{a}hler of dimension $2n$. Using the Hirzebruch-Riemann-Roch formula, this gives a relation between the Chern  numbers of $X$, but we can express it more simply using the $\alpha_I$ since
$\chi(S^{[I]},\mathcal{O}_{S^{[I]}})=(n_1+1)\ldots(n_k+1)$ for the partition $I$ of $n$ given by $n=n_1+\ldots +n_k$.
The relation is thus
\begin{eqnarray} \label{eqaffine} n+1=\sum_I\alpha_I(n_1+1)\ldots(n_k+1)
\end{eqnarray}
and similarly for the $\beta_I$.
For example, in dimension $4$, the Hirzebruch-Riemann-Roch formula  provides the relation (see \cite{ogradyccm})
\begin{eqnarray} \label{eqchitop3} 3=\frac{1}{240}(c_2(X)^2-\frac{1}{3}c_4(X)),
\end{eqnarray}
while in our setting, it writes 
$$3\alpha_{2}+4\alpha_{1,1}=3.$$
In the case of dimension $4$ we have two topological models, the Hilbert scheme $S^{[2]}$ and the generalized Kummer variety $K_2(A)$ and they clearly have independent classes, since otherwise by (\ref{eqaffine}) their classes would be equal, hence also their topological  Euler-Poincar\'{e} characteristic $c_4$, which is not the case. In dimension  $6$, we have
$3$ topological models, namely $S^{[3]},\,K_3(A)$ and ${\rm OG6}$ constructed in \cite{ogrady}, and
 their classes are linearly  independent, as proves the following computation.
 The Chern numbers $c_2^3,\,c_2c_4,\,c_6$ of $K3^{[3]}$ are computed in \cite{EGL},  those of $K_3(A)$ are computed in \cite{Nieperkummer}, and those of ${\rm OG}6$ are computed in \cite{MongardiRapSac}. Thanks to these works, the matrix of Chern numbers for these three varieties takes the form  (where the first line indicates the Chern numbers of $K3^{[3]}$, the second line those of $K_3(A)$, and the third line those of ${\rm OG}6$):
 $$\begin{pmatrix}
 36800&14720  & 3200\\
 30208&6784 & 448\\
30720 &7680 & 1920
\end{pmatrix} .$$
The determinant of this matrix is nonzero, proving the independence of the three classes.
 Thus, up to dimension $3$,  the classes of hyper-K\"{a}hler manifolds  generate the affine space defined by (\ref{eqaffine}). It is likely that there are   linear relations in higher dimension.

Other contraints are given by inequalities. For example, the class $c_2$ has positivity properties related to the existence of K\"{a}hler-Einstein metrics. Positivity results for some Chern numbers have been also obtained by Jiang \cite{jiangchen} who proves that the coefficients of the Riemann-Roch polynomial of a line bundle $L$ on $X$, expressed as a polynomial in $q(L)$, has positive coefficients. It is proved in  \cite{Nieper} that for an adequate  normalization of the Beauville-Bogomolov form $q$, these coefficients are given by  Chern numbers of $X$ (depending only on the dimension).  In dimension $4$, work of Guan \cite{guan} gives inequalities on $c_4(X)$  that come from the study of the cohomology algebra of $X$. In higher dimension $2n$, work of \cite{GLetco} also predicts bounds on Betti numbers which in turn gives conjectural bounds on the topological Euler-Poincar\'{e} characteristic $c_{2n}(X)$. 
It would be very interesting to have an idea of the convex set generated by classes of hyper-K\"{a}hler manifolds.  Let us now mention three specific questions in this direction. 

\vspace{0.5cm}

(a) {\bf The numbers $\chi(X,\Omega_X^i)$.} In the case of  the varieties  $S^{[n]}$ and ${\rm Kum}_n(A)$, we have the following result.
\begin{lemm}\label{lechiincrease} Let $S$ be a $K3$ surface. Then the numbers
$(-1)^i\chi(S^{[n]},\Omega_{S^{[n]}}^i)$ are increasing in the range $0\leq i\leq n$.

Similarly, for $n$ fixed, the numbers $(-1)^i\chi({\rm Kum}_n(A),\Omega^i_{{\rm Kum}_n(A)})$ are increasing.
\end{lemm}
\begin{proof} We argue as in Section \ref{secsmalldim}. As these numbers are Chern numbers by the Hirzebruch-Riemann-Roch formula, we can replace by \cite{EGL} the $K3$ surface $S$ by the disjoint union $\Sigma$ of two copies of $\mathbb{P}^2$ blown-up in $9$ points. Then $(-1)^i\chi(\Sigma^{[n]},\Omega^i_{\Sigma^{[n]}})=b_{2i}(\Sigma^{[n]})$ so the statement is that $b_{2i}(\Sigma^{[n]})$ is increasing in the range $0\leq i\leq n$ and this follows from the hard Lefschetz theorem since ${\rm dim}\,\Sigma^{[n]}=2n$.

  For the second statement, the numbers $(-1)^i\chi({\rm Kum}_n(A),\Omega^i_{{\rm Kum}_n(A)})$  are  computed in \cite{gottschesoergel}  which gives the following formula
$$\sum_i(-1)^i\chi({\rm Kum}_n(A),\Omega^i_{{\rm Kum}_n(A)})y^i = n \sum_{d|n} d^3 (1 + y + ... + y^{n/d-1})^2 y^{n - n/d}, $$
from which it immediately follows that these numbers are increasing in the range $0\leq i\leq n$.  
\end{proof}
We also computed these  numbers for ${\rm OG}6$ and ${\rm OG}10$ and got
$$
 (-1)^i\chi({\rm OG}6,\Omega_{{\rm OG}6}^i)=\,\,4,\, 24,\, 348,\,1168$$  respectively for $i=0,\,1,\,2,\,3$ and 
 $$
 (-1)^i \chi({\rm OG}10,\Omega_{{\rm OG}10}^i)=6, \,111,\, 1062, \,7173,\, 33534, \,93132,
 $$
  respectively for $i=0,\,1,\,2,\,3,\,4,\,5$.
 In the two cases, these numbers are increasing.
 This raises the following question.
 \begin{question}
 Is it true that the numbers $(-1)^i\chi(X,\Omega_X^i)$ are increasing in the range $0\leq i\leq n$ for any hyper-Kähler manifold $X$ of dimension $2n$?
 \end{question}

 \vspace{0.5cm}
 
 (b) {\bf Chern character numbers}. Theorems \ref{thm:hilb case} and \ref{thm:Kum case} prove that the two numbers $(-1)^n\int_{S^{[n]}}{\rm ch}_{2n}(S^{[n]})$  and $(-1)^n\int_{{\rm Kum}_n(A)}{\rm ch}_{2n}({\rm Kum}_n(A))$ are positive for any $n$. 
 
 This suggests the following question.
 \begin{question} Is it true that  $(-1)^nM(X)=(-1)^n\int_X{\rm ch}_{2n}(X)$ is positive for any hyper-Kähler manifold $X$ of dimension $2n$?
  \end{question}
  The following lemma gives an affirmative answer in dimension $4$.
  \begin{lemm}\label{lech4} Let $X$ a a hyper-Kähler fourfold. Then $M(X)=\int_X{\rm ch}_{4}(X)>0$.
  \end{lemm}
  \begin{proof} We have ${\rm ch}_{4}(X)=\frac{1}{24}( 2 c_2^2(X) - 4 c_4(X) )$ so the statement is equivalent to $\int_X(c_2^2(X) - 2 c_4(X))>0$.
  Formula (\ref{eqchitop3}) gives us  $\int_Xc_2(X)^2=720+\frac{1}{3}\int_Xc_4(X)$, so the desired inequality is equivalent to 
  \begin{eqnarray}\label{eqinec4} \int_Xc_4(X)=\chi_{\rm top}(X) <\frac{9\cdot 240}{5}=432.
  \end{eqnarray}
  Inequality (\ref{eqinec4}) now follows from work of Salamon \cite{salamon} and Guan \cite{guan}. By \cite{salamon},
  $b_3(X)+b_4(X)=46+10 b_2(X)$, 
  hence $\chi_{\rm top}(X)=b_4(X)-2b_3(X)+2b_2(X)+2\leq 48+12 b_2(X)$. Guan proves that $b_2(X)\leq 23$, so we get 
  $$ \chi_{\rm top}(X)\leq 48+12\cdot 23=324,$$
  proving (\ref{eqinec4}).
  \end{proof}
  
 Proposition \ref{prop:positivity} shows that  $(-1)^n \int_{{\rm Kum}_{n}(A)} {\rm ch}_{2k_1} .... {\rm ch}_{2k_r}>0$ for any choice of partition $n=\sum_i k_i$. This suggests the following question
 
 \begin{question} \label{questioncherncharacternumber} Is it true that  $(-1)^n\int_X{\rm ch}_{2k_1} .... {\rm ch}_{2k_r}$ is positive for any hyper-Kähler manifold $X$ of dimension $2n$ and  any choice of partition $n=\sum_i k_i$?
  \end{question}
 
 \vspace{0.5cm}
 
 (c)  {\bf Positivity of Chern numbers}. We recall here for completeness  that   positivity properties  had been observed already in \cite{nieperthese}, \cite{sawon} for the monomial Chern numbers $\int_Xc_{2k_1}(X)\ldots c_{2k_r}(X)$ of known hyper-Kähler manifolds. The following question was asked in 
 \cite{nieperthese}
 \begin{question}\label{qusetionchernnumber}
 Is it true that  $(-1)^n\int_Xc_{2k_1}(X) .... c_{2k_r}(X)$ is positive for any hyper-Kähler manifold $X$ of dimension $2n$ and  any choice of partition $n=\sum_i k_i$?
 \end{question}
 We note that, in the case of dimension $4$, it is still unknown that $\chi_{\rm top}(X)=\int_Xc_4(X)>0$. The questions (b) and (c) look very similar but they lead to very different convexity inequalities and,  in dimension $4$, the two inequalities $\int_Xc_4(X)>0$ (conjectured above) and  $\int_X{\rm ch}_4(X)>0$ proved in Lemma \ref{lech4} imply together the finiteness of the complex cobordism classes of hyper-Kähler fourfolds.
 
 \vspace{0.5cm}
 
We finish with two  questions more specifically related to our results, concerning the comparison of the two systems of linear generators $S^{[I]}$ and ${\rm Kum}_I(A)$. It would be interesting to know more about the matrix comparing these two systems of linear generators  in each dimension. 
\begin{question}\label{questionunpeuvague} Is there a geometric way of understanding and  computing this matrix?
\end{question}
Another intriguing fact concerns the shape of the coefficients of these matrices. 
Since the Chern numbers of $S^{[k]}$ and $\Kum_{k}(A)$ are known for small values of $k$, and the Chern numbers of a product $X\times Y$ can be expressed in terms of Chern numbers of $X$ and $Y$, one gets consequently the Chern numbers of $S^{[I]}$ and $\Kum_I(A)$ for all partitions $I$ of $k$. One may then study the linear relations among the classes of these manifolds. Below is the explicit expression giving the class of 
$S^{[k]}$  as a $\mathbb{Q}$-linear combination of the classes of ${\rm Kum}_I(A)$ for $k\leq 5$. 
\begin{equation}
\label{eqexpressionSenA}
\begin{aligned}  
S^{[2]} & = 1/3 \Kum_{2}(A) + 1/2 \Kum_{1,1}(A)\\  
S^{[3]} & = 1/5 \Kum_3(A) + 14/45 \Kum_{2,1}(A) + 1/6 \Kum_{1,1,1}(A)\\  
S^{[4]} & = 1/7 \Kum_{4}(A) + 7/40 \Kum_{3,1}(A) + 1/21 \Kum_{2,2}(A)\\  & + 47/315 \Kum_{2,1,1}(A) + 1/24 \Kum_{1,1,1,1}(A)\\
S^{[5]} & = 1/9 \Kum_{5}(A) + 62/525 \Kum_{4,1}(A) + 4/75 \Kum_{3,2}(A) + 49/600 \Kum_{3,1,1}(A)\\ 
& + 23/525 \Kum_{2,2,1}(A) + 151/3150 \Kum_{2,1,1,1}(A) + 1/120 \Kum_{1,1,1,1,1}(A).
\end{aligned}
\end{equation}
The leading coefficient being $\frac1{2k-1}$ can be explained by the difference in the expression of Milnor genus for the two infinite series, since the other terms are products and do not contribute to the Milnor genus.

Similarly, we computed the class of 
${\rm Kum}_k(A)$  as a $\mathbb{Q}$-linear combination of the classes of $S^{[I]}$ for $k\leq 5$.
\begin{equation}
\label{eqexpressionAenS} 
\begin{aligned}
{\rm Kum}_2(A) & = 3 S^{[2]} - 3/2 S^{[1,1]}\\
{\rm Kum}_3(A) & = 5 S^{[3]} - 14/3 S^{[2,1]} + 3/2 S^{[1,1,1]}\\
{\rm Kum}_4(A) & = 7 S^{[4]} - 49/8 S^{[3,1]} - 3 S^{[2,2]} + 67/12 S^{[2,1,1]} - 21/16 S^{[1,1,1,1]}\\
{\rm Kum}_5(A) & = 9 S^{[5]} - 186/25 S^{[4,1]} - 36/5 S^{[3,2]} + 1287/200 S^{[3,1,1]} \\
& +159/25 S^{[2,2,1]} - 577/100 S^{[2,1,1,1]} + 423/400 S^{[1,1,1,1,1]}.
\end{aligned}
\end{equation}
Equations (\ref{eqexpressionSenA}) strongly suggest the following question.
\begin{question}\label{questionconvexity} Is it true that for any $n$, the class of $S^{[n]}$ is a linear combination with positive coefficients of the classes of ${\rm Kum}_I(A)$?
\end{question}
There are only two known hyper-Kähler manifolds which do not belong to the two infinite series discussed above, namely the $6$-dimensional and $10$-dimensional O'Grady manifolds ${\rm OG}_6$ and ${\rm OG}_{10}$  (see \cite{ogrady}, \cite{og10}).
Their cobordism  classes are expressed as follows in the generalized Kummer basis (showing in particular that not any hyper-Kähler manifold has its class in the convex cone generated by products of generalized Kummer varieties).
\begin{align*}
{\rm OG}_6  & = 6/5 {\rm Kum}_{3}(A) - 16/45 {\rm Kum}_{2,1}(A) + 1/6 {\rm Kum}_{1,1,1}(A),\\ \nonumber
{\rm OG}_{10} & = 25/168 {\rm Kum}_5(A) + 67/700 {\rm Kum}_{4,1}(A) + 3/700 {\rm Kum}_{3,2}(A) + 163/1600 {\rm Kum}_{3,1,1}(A) \\
& \nonumber+ 2617/37800 {\rm Kum}_{2,2,1}(A) + 493/12600 {\rm Kum}_{2,1,1,1}(A) + 17/1920 {\rm Kum}_{1,1,1,1,1}(A).
\end{align*}

Our last observation is the following. There is a mysterious link (in fact related to mirror symmetry)  between hyper-K\"{a}hler manifolds of dimension $2n$ and rational homology projective space $\mathbb{C}\mathbb{P}^n$. It  appears for example  in \cite{KLSV} where it is proved that the dual complex of the central fiber  of a maximally unipotent dlt degeneration of a hyper-K\"{a}hler $2n$-fold is a rational homology projective space $\mathbb{C}\mathbb{P}^n$. There is another mysterious and more precise  link between ${\rm K3}^{[n]}$ and  projective space $\mathbb{P}^n$, which comes from the study of the Riemann-Roch polynomials. Indeed, one has the following result that can be formulated using the Chern numbers of $X$ by \cite{Nieper}. (This result is used by looking at the natural  Lagrangian fibration of a variety $S^{[n]}$ where $S$ is a $K3$ surface  equipped with an elliptic fibration.)
\begin{theo} \cite{EGL} Let $X$ be a hyper-K\"{a}hler manifold of $K3^{[n]}$-deformation type and $q$ be its Beauville-Bogomolov form. Then for any line bundle $L$ on $X$ with $q(c_1(L))=2k$, one has $\chi(X,L)=\chi(\mathbb{P}^n,\mathcal{O}_{\mathbb{P}^n}(k+1))=h^0(\mathbb{P}^n,\mathcal{O}_{\mathbb{P}^n}(k+1))$.
\end{theo}

The formalism used in the present paper proposes a further analogy between  ${\rm K3}^{[n]}$ and $\mathbb{P}^n$. Namely the classical complex cobordism  gives the  projective spaces $\mathbb{P}^n$ as multiplicative rational  generators of ${\rm MU}^*({\rm pt})$ while we proved that the   ${\rm K3}^{[n]}$ are multiplicative rational   generators of ${\rm MU}^*({\rm pt})_{\rm even}$.

Mathemathisches Institut, Universit\"at Bonn

georgo@math.uni-bonn.de \\

Université de Paris, CNRS, IMJ-PRG

jieao.song@imj-prg.fr \\

CNRS, Institut de Math\'{e}matiques de Jussieu-Paris rive gauche

claire.voisin@imj-prg.fr
    \end{document}